\documentclass[14pt,a4paper,twoside]{amsart}
\usepackage{mathpazo}
\usepackage{amssymb}
\usepackage{mathtools}
\usepackage{parskip}
\usepackage{amsthm}
\usepackage{cite}
\usepackage{enumerate}
\usepackage{hyperref}
\usepackage{cleveref}
\usepackage{csquotes}
\usepackage{hyperref}
\usepackage{tikz}
\usepackage{tikz-cd}
\usepackage{xcolor}
\hypersetup{
	colorlinks=true,
	linkcolor=blue,
	filecolor=blue,
	urlcolor=blue,
	citecolor=blue,
	pdfpagemode=Fullscreen
}
\newtheorem{theorem}{Theorem}[section]
\newtheorem{lemma}[theorem]{Lemma}
\newtheorem*{Acknowledgement}{\textnormal{\textbf{Acknowledgement}}}
\theoremstyle{definition}

\newtheorem{corollary}[theorem]{Corollary}
\newtheorem{proposition}[theorem]{Proposition}
\newtheorem{remark}[theorem]{Remark}

\numberwithin{equation}{section}

\title{Derivations of plane algebroid curves}

\date{}
\author{Sagnik Chakraborty$^1*$}	
\address{$^1$Department of Mathematics, 
	Ramakrishna Mission Vivekananda Educational and Research Institute, 
	Belur Math,  Howrah - 711 202,
	West Bengal, India.}	
\email{jusagnik28@gmail.com}

\author{R. V. Gurjar$^2$}
\address{$^2$Department of Mathematics, IIT Bombay, Powai, Mumbai – 400 076, Maharashtra, India.}
\email{rv.gurjar50@gmail.com} 

\author{Madhuparna Pal$^3$}
\address {$^3$Department of Mathematics,
	Ramakrishna Mission Vivekananda Educational and Research Institute, 
	Belur Math,  Howrah - 711 202,
	West Bengal, India.}	
\email{honey.madhu777p@gmail.com}

\begin{document}

	\maketitle
\iffalse	
	\author{Sagnik Chakraborty$*$}\\	
	\address{Department of Mathematics, 
		Ramakrishna Mission Vivekananda Educational and Research Institute, 
		Belur Math,  Howrah - 711 202,
		West Bengal, India.}	
	\email{Email: jusagnik28@gmail.com}
	
	\author{R. V. Gurjar}\\
	\address{Department of Mathematics, IIT Bombay, Powai, Mumbai – 400 076, Maharashtra, India.}
	\email{Email: rv.gurjar50@gmail.com} 
	
	\author{Madhuparna Pal}\\
	\address {Department of Mathematics,
		Ramakrishna Mission Vivekananda Educational and Research Institute, 
		Belur Math,  Howrah - 711 202,
		West Bengal, India.}	
	\email{Email: honey.madhu777p@gmail.com}
\fi	
	
	\begin{abstract}
    In this paper, we study the derivations of an irreducible plane algebroid curve $R:=\frac{k[[X,Y]]}{(f)}$, which is not regular. Since the normalization of $R$ is isomorphic to $k[[t]]$, every $k$-derivation of $R$ is induced from a $k$-derivation of $k[[t]]$, which are of the form $a(t)\frac{d}{dt}$ for some $a(t)\in k[[t]]$. We establish a lower bound on the $t$-order of a nonzero power series $a(t)$ for $a(t)\frac{d}{dt}$ to induce a derivation of $R$ in terms of the multiplicity of $R$. We also prove a related result for the value semi-group of such curves.

	\end{abstract}

\footnote{$*$ corresponding author}
    
	\iffalse
%
	
	\fi
	
	\section{Introduction}
	
	Let $ k $ be an algebraically closed field of characteristic zero, and let $k[[X,Y]]$ be the formal power series ring in two variables over $k$. By an irreducible {\it plane algebroid curve} (over $k$), we mean the quotient ring $R \coloneqq \frac{k[[X, Y]]}{(f)}$, where $f\in k[[X,Y]]$ is an irreducible power series. Throughout this article, we assume that $ R $ is not regular, i.e., the order of $ f $ is at least two. The integral closure of $R$ in its field of fractions, denoted by $\overline{R}$, is a complete {\it discrete valuation ring}, which is isomorphic to the power series ring $k[[t]]$, where $t$ corresponds to a {\it uniformizing parameter} of $\overline R$. Let $\mathit v$ be the valuation of $\overline R$. The {\it $t$-order} of a nonzero element $r\in R$, denoted by $\text{ord}_t(r)$, is defined as the smallest power of $t$, which appears in $r$ when we write it as a power series in $t$. Note that the $t$-order of a nonzero element $r$ is the same as $v(r)$.\\
    If we write $R=k[[x,y]]$, where $x,y$ are the residue classes of $X,Y$, respectively, then, without loss of generality, we may assume that $n:=v(x)\le v(y)$. Since $R$ is not regular, $n>1$. Replacing $y$ with $y-g(x)$ for a suitable polynomial $g$, we can further assume that $v(y)=m$ for some $m>n$ which is not a multiple of $n$. Therefore, after taking a suitable automorphism of $\overline R=k[[t]]$ and replacing $y$ with a nonzero scalar multiple of $y$, we can write $x,y$ as $x=t^n$ and $y=t^m+\textit{higher-degree terms}$. This gives a parametric representation of $R$ as $R=k[[t^n,t^m+\textit{higher degree-terms}]]$ in $\overline{R}=k[[t]]$, which is called a {\it Puiseux parametric representation} of $R$.\\
    %It is well-known that $\text{Der}_kR$, {\it the module of $k$-derivations of $R$, is generated by two elements as an $ R $-module}. There are several proofs of this result; for example, see Theorem 1.1 of \cite{Ps}.
    The {\it conductor ideal} of $ R $, denoted by $\mathfrak C_R$, is of the form $t^ck[[t]]$ for some positive integer $c >1$. In particular, $t^{c+i}\in R$ for every non-negative integer $i$. By a result of A. Seidenberg (see the Theorem on page 168 of \cite{Se}), every $k$-derivation of $R$ uniquely extends to a $k$-derivation of $\overline{R}= k[[t]]$. Now, it is easy to see that every $k$-derivation of $k[[t]]$ is of the form $a(t)\frac{d}{dt}$ for some $a(t)\in k[[t]]$. By definition, $a(t)\frac{d}{dt}$ induces a $k$-derivation of $R$ iff $a(t)\frac{dr}{dt}\in R$ for all $r\in R$; or equivalently, iff $a(t)\frac{dx}{dt},a(t)\frac{dy}{dt}\in R$. Therefore, describing all $k$-derivations of $R$ is equivalent to finding all power series $a(t)\in k[[t]]$ such that $a(t)\frac{dx}{dt},a(t)\frac{dy}{dt}\in R$.\\ 
    Since $R$ contains $t^{c+i}$ for all nonnegative integers $i$, it is easy to see that if the $t$-order of $a(t)$ is at least $c-n+1$, then the $k$-derivation $a(t)\frac{d}{dt}:k[[t]]\to k[[t]]$ induces a $k$-derivation of $R$. However, if the $t$-order of $a(t)$ is smaller than $c-n+1$, then, in general, it is difficult to determine whether $a(t)\frac{d}{dt}$ induces a $k$-derivation of $R$ or not. Since $n>1$, if the constant term of $a(t)$ is nonzero, then $a(t)\frac{d}{dt}(t^n)\notin R$. Therefore, if $a(t)\frac{d}{dt}$ induces a nonzero $k$-derivation of $R$, then the $t$-order of $a(t)$ must be at least $1$. We are interested in finding the minimum element of $\{\text{ord}_t(a)\ |\ a(t)\frac{d}{dt}\text{ induces a nonzero $k$-derivation of $R$}\}$.\\
    If we can write $x=t^n$ and $y=t^m$ (where $m,n$ are relatively prime due to the irreducibility of $f$), then $t\frac{d}{dt}$ induces a $k$-derivation of $R$. By a slight abuse of terminology, we will call such irreducible power series $f$ {\it quasi-homogeneous} (see {\it preliminaries} for the definition). The main result of this paper (Theorem \ref{main}), proved in \cref{sec:3}, is that {\it if an irreducible power series $ f \in k[[X,Y]]$ of order $n>1$ is not quasi-homogeneous, then for every $k$-derivation $a(t)\frac{d}{dt}$ of $ R $, the $t$-order of $ a(t)$ is at least $n+1$.} %Note that if $f = X^n - Y^m$, where $n<m$ are relatively prime positive integers, then $t\frac{d}{dt}$ is a $k$-derivation of $R$. This explains why we assume that $f$ is not quasi-homogeneous. 
    Our result shows that $ R $ does not admit any $ k $-derivation of the form $a(t)\frac{d}{dt}$, where the $t$-order of $a(t)$ is small, unless $ f $ is quasi-homogeneous. In particular, it shows that if $R$ is not regular and $t\frac{d}{dt}$ induces a derivation of $R$, then $f$ must be quasi-homogeneous.\\
    Recall that the {\it value semi-group} of $R$, denoted by $v(R)$, is defined as $v(R):=\{v(r) \ |\ r\in R\setminus\{0\}\}$. This is an important tool for studying the nature of plane curve singularities. Note that $v(R)$ is a submonoid of $(\mathbb N,+)$, the set of all non-negative integers under addition. Since $t^{c+i}\in R$ for all $i\ge 0$, $\mathbb N\setminus v(R)\subseteq \{1,2,\ldots,c-1\}$ is a finite set.\\ 
    We prove a result about the value semi-group of a plane algebroid curve in \cref{sec:4}. O. Zariski proved in \cite[Theorem 4]{1} that {\it if $f$ is not quasi-homogeneous, then $f$ has a parametric representation of the form $x=t^n$, $y=t^m+bt^\lambda+\textit{higher-degree terms}$, where $b\in k$ is nonzero and $\lambda>m$ is such that $\lambda+n\notin v(R)$}. We show that (Corollary \ref{4cor}) {\it in a parametric representation of $R$ as above, if we, moreover, assume that $v(R)$ is not generated by $m,n$ and $\lambda$ is not a multiple of $\text{g.c.d.}(m,n)$, then $\lambda+2n\notin v(R)$}. It is used to construct a family of plane algebroid curves for which the inequality proved in \cref{sec:3} is a strict inequality (Proposition \ref{4prop2}). Finally, we state an open problem that could pave the way to further exploration of the nature of derivations of plane algebroid curves.%Note that if $v(R)$ is not generated by $m$ and $n$, then $f$ is not quasi-homogeneous.
	
	\section{Preliminaries}

\iffalse

\fi

By a field $k$, we mean an algebraically closed field $k$ of characteristic zero. We use $k[X_1,\dots, X_n]$ (respectively, $k[[X_1,\dots,X_n]]$) to denote the polynomial ring (respectively, the formal power series ring) in $n\ge 2$ variables over $k$.
%It is well-known that $k[X_1,\dots,X_n]$ and $k[[X_1,\dots,X_n]]$ are Noetherian unique factorization domains of (Krull) dimension $n$, and the latter one is a regular local ring. 
We will also use the symbol $k^{[n]}$ (respectively, $k^{[[n]]}$) to denote the polynomial ring (respectively, the formal power series ring) in $n$ variables over $k$, especially in situations where we do not want to specify the variables at the outset. The unique maximal ideal of $k^{[[n]]}$ will be denoted by $\mathfrak M_n$. A finite sequence of elements $X_1,\dots,X_n\in \mathfrak M_n$ is called a {\it set of variables} if $k^{[[n]]}=k[[X_1,\ldots,X_n]]$. 
%It is well-known that every $k$-algebra automorphism of $k^{[[n]]}$ maps a set of variables to a set of variables; and if $X_1,\dots,X_n\in k^{[[n]]}$ is a set of variables, then $X'_1,\dots,X'_n\in k^{[[n]]}$ is another set of variables iff there exists a $k$-algebra automorphism $\sigma:k^{[[n]]}\to k^{[[n]]}$ such that $\sigma(X_i)=X'_i$ for $1\le i\le n$, or equivalently, iff there exists an $n\times n$ invertible matrix $\mathtt M:=(\mathtt M_{ij})_{1\le i,j\le n}$ over $k$ such that $X'_i-\sum_{j=1}^n\mathtt M_{ij}X_j\in \mathfrak M_n^2$ for all $i=1,\dots,n$.\\
The {\it order} of a nonzero element $f\in k^{[[n]]}$ is defined as the largest non-negative integer $i$ such that $f\in \mathfrak M_n^i$. This is also called the {\it multiplicity} of the quotient ring $\frac{k^{[[n]]}}{(f)}$. Note that a $k$-algebra automorphism of $k^{[[n]]}$ does not change the orders of its nonzero elements. If $X_1,\dots,X_n$ is a set of variables of $k^{[[n]]}$, then every nonzero element $f\in k^{[[n]]}$ can be uniquely written as an infinite formal sum $f=\sum_{j=0}^\infty f_j$, where each $f_j$ is either zero or a homogeneous polynomial of degree $j$ in $X_1,\dots,X_n$. Clearly, the order of $f$ is the smallest non-negative integer $i$ such that $f_i\neq 0$.\\%equal to the degree of the homogeneous polynomial of the smallest degree that appears in $f$.\\
%Recall that $U,V\in k[[X,Y]]$ is called {\it a pair of variables} if there exists a $k$-algebra automorphism $\sigma:k[[X,Y]]\to k[[X,Y]]$ such that $\sigma(X)=U$ and $\sigma(Y)=V$. It is well-known that $U,V$ is a pair of variables iff we can (uniquely) write $U,V$ as $U=aX+bY+U'$ and $V=cX+dY+V'$, for some $U',V'\in \mathfrak m^2$ and $a,b,c,d\in k$ such that $ad-bc=1$; or equivalently, iff $U,V\in \mathfrak m$ and $k[[X,Y]]=k[[U,V]]$. We will also use the symbol $k^{[2]}$ (respectively, $k^{[[2]]}$) to denote the polynomial ring (respectively, the formal power series ring) in two variables over $k$, especially in situations where we do not intend to specify the variables at the outset.\\ %The field of complex numbers is denoted by $\mathbb C$, and we use $\mathbb C\{x,y\}$ to denote the convergent power series ring in two variables over $\mathbb C$.\\
Recall that a nonzero polynomial $f\in k[X,Y]$ is called a {\it quasi-homogeneous polynomial} (in $X,Y$, with {\it weights} $a,b$, which are positive integers) if $f(X^a,Y^b)$ is a homogeneous polynomial in $X,Y$. We state a standard result for quasi-homogeneous polynomials in two variables, leaving its proof as an exercise for the interested reader.
\begin{lemma}
\label{quasi-hom}
    Let $k$ be an algebraically closed field of characteristic zero. Let $f\in k[X,Y]$ be an irreducible polynomial whose order, as an element of $k[[X,Y]]$, is at least two. If $f$ is quasi-homogeneous, then it is a binomial, i.e., there exist (relatively prime) positive integers $m,n\ge 2$ such that $f=Y^n-X^m$.
\end{lemma}
If $f\in k^{[[2]]}$ is an irreducible power series of order $>1$, then the plane algebroid curve $R:=\frac{k^{[[2]]}}{(f)}$ is called {\it quasi-homogeneous} if $R$ is isomorphic to $k[[t^n,t^m]]$ as a $k$-algebra for some relatively prime positive integers $m,n\ge 2$; or equivalently, if there exists a pair of variables $X,Y\in k^{[[2]]}$ such that $(f)=(Y^n-X^m)$ for some positive integers $m,n\ge 2$ (Since $f$ is irreducible, $m$ and $n$ must be relatively prime.). In view of {\it lemma \ref{quasi-hom}}, $R$ is quasi-homogeneous iff the ideal generated by $f$ is the same as the ideal generated by a power series $g\in k^{[[2]]}$, which is a quasi-homogeneous polynomial with respect to a pair of variables of $k^{[[2]]}$.\\ 
By a slight abuse of terminology, we call an irreducible power series $f\in k^{[[2]]}$ of order $>1$ {\it quasi-homogeneous} if the associated plane algebroid curve $R:=\frac{k^{[[2]]}}{(f)}$ is quasi-homogeneous.
%quasi-homogeneity is a property of the plane algebroid curve $R$, and not of $f$. The power series $f$ may not be a quasi-homogeneous polynomial even \\

\begin{remark}
    If $f$ is a quasi-homogeneous polynomial, in the usual sense, then it is quasi-homogeneous according to our definition, but the converse is false. For example, if we consider the irreducible power series $f(X,Y):=Y^3-3X^3Y-X^4-X^5\in k[[X,Y]]$ then, clearly, $f$ is not a quasi-homogeneous polynomial. But $R=k[[x,y]]:=\frac{k[[X,Y]]}{(f)}\cong k[[t^3,t^4+t^5]]$, and using the ideas in the proof of \cite[Theorem 4]{1}, we can show that there exists a power series $\tau\in k[[t]]$ of $t$-order $1$ such that $k[[x,y]]=k[[\tau^3,\tau^4]]$, which implies that $f$ is quasi-homogeneous according to our definition.
\end{remark}

\par If $A$ is a $k$-algebra, then a {\it $k$-linear derivation}, or, a {\it $k$-derivation}, $D$ is defined as a $k$-linear map $D:A\to A$ that satisfies the {\it Leibniz rule}, i.e., $D(ab)=aD(b)+bD(a)$ for all $a,b\in A$. If the underlying field $k$ is understood from the context, then a $k$-linear derivation of $A$ will often simply be called a {\it derivation} of $A$. It is well-known that if $A=k^{[[2]]}$ and $X,Y\in A$ is a pair of variables, then a derivation $D$ of $A$ can be uniquely written as $D=D(X)\frac{\partial}{\partial X}+D(Y)\frac{\partial}{\partial Y}$, where $\frac{\partial}{\partial X},\frac{\partial}{\partial Y}$ are the usual formal partial derivatives with respect to $X,Y$, respectively. If $f\in A$, it is customary to use $f_X$ and $f_Y$ to denote $\frac{\partial f}{\partial X}$ and $\frac{\partial f}{\partial Y}$, respectively, so that $D(f)=f_XD(X)+f_YD(Y)$. If $U,V\in \mathfrak M_2$ is another pair of variables, then applying the usual {\it chain rule} for derivation, we get \[f_U=f_X\frac{\partial X}{\partial U}+f_Y\frac{\partial Y}{\partial U} \ \text{  and  }\  f_V=f_X\frac{\partial X}{\partial V}+f_Y\frac{\partial Y}{\partial V}\ .\]
Since $\frac{\partial X}{\partial U}\cdot\frac{\partial Y}{\partial V}-\frac{\partial Y}{\partial U}\cdot\frac{\partial X}{\partial V}$ is a unit of $k^{[[2]]}$, it follows that $(f_X,f_Y)=(f_U,f_V)$. We record this observation in the following lemma.
\begin{lemma}
\label{partial}
Let $k$ be a field, and $f\in k^{[[2]]}$. If $X,Y$ and $U,V$ are two pairs of variables of $k^{[[2]]}$, then $(f_X,f_Y)=(f_U,f_V)$, i.e., the ideal generated by $f_X$ and $f_Y$ is the same as the ideal generated by $f_U$ and $f_V$.    
\end{lemma}

Let $f\in k^{[[2]]}$ be an irreducible power series of order $>1$. Since the plane algebroid curve $R:=\frac{k^{[[2]]}}{(f)}$ is a one-dimensional complete local domain containing $k$, it follows from {\it Cohen's structure theorem} for complete local rings that $\overline{R}$, the integral closure of $R$ in its field of fractions, is isomorphic to the power series ring $k[[t]]$.\\ The {\it conductor ideal} of $ R $, denoted by $\mathfrak C_R$, is defined as $\mathfrak C_R:=\text{Ann}_R\big(\frac{\overline{R}}{R}\big)$. It is the largest ideal of $R$ which is also an ideal of $\overline{R}$. Since $R$ is not regular and $\mathfrak C_R$ is an ideal of $\overline R\cong k[[t]]$, it must be of the form $t^ck[[t]]$ for some positive integer $c>1$, which is called the {\it length of the conductor ideal} $\mathfrak C_R$. %Note that if $v(R)$ denotes the value semi-group of $R$, then $c$ is the smallest non-negative integer such that $c+\mathbb N\subseteq v(R)$, where $\mathbb N$ is the set of non-negative integers. 
If we denote the length of an $R$-module $M$ by $l_R(M)$, then $c=l_R\big(\frac{\overline{R}}{\mathfrak C_R}\big)$.\\
There exists a short exact sequence of $R$-modules \[0 \to \frac{R}{\mathfrak C_R} \to \frac{\overline R}{\mathfrak C_R} \to \frac{\overline R}{R} \to 0.\] 
Since length is additive, $c = l_R\big(\frac{R}{\mathfrak C_R}\big) + l_R\big(\frac{\overline R}{R}\big)$. It is well-known that $l_R\big(\frac{\overline R}{\mathfrak C_R}\big)=2l_R\big(\frac{R}{\mathfrak C_R}\big)$ (see, for example, \cite{Kunz}), which implies that $l_R\big(\frac{\overline R}{R}\big) = l_R\big(\frac{R}{\mathfrak C_R}\big) = \frac{c}{2}$. Let $J_f$ be the {\it Jacobian ideal} of $f$, i.e., the ideal of $k^{[[2]]}$ generated by the partial derivatives of $f$ (By {\it lemma \ref{partial}}, $J_f$ does not depend on the choice of a pair of variables.). The {\it Milnor number} of $f$, introduced by J. Milnor, is defined as $\mu(f):=\text{dim}_k\big(\frac{k^{[[2]]}}{J_f}\big)$. Since $k$ is an algebraically closed field of characteristic zero and $f$ is irreducible, it follows from a result of J. J. Risler \cite[Theorem 1]{Ris}, which was earlier proved by H. W. E. Jung in \cite{Jung}, that $\mu(f)=2l_R\big(\frac{\overline R}{R}\big)=c$. Note that if $\sigma$ is a $k$-algebra automorphism of $k^{[[2]]}$, then $R:=\frac{k^{[[2]]}}{(f)}\cong \frac{k^{[[2]]}}{(\sigma(f))}$, which implies that $\mu(f)=\mu(\sigma(f))$.\\ 
Following \cite{1}, by abuse of notation, we use $\Omega_{R/k}$ to denote the {\it module of universally finite K\"ahler differentials} of $R$ over $k$ (see \cite[Chapters 11--13]{Ku0} for the definition and basic properties of the module of universally finite K\"ahler differentials). As in \cite{1}, we denote by $T$ the torsion submodule of $\Omega_{R/k}$. Then \cite[Theorem 1]{1}, together with Risler's result, implies
\[l_R(T)=l_R\bigg(\frac{k^{[[2]]}}{(f,J_f)}\bigg)\le l_R\bigg(\frac{k^{[2]]}}{J_f}\bigg)=\mu(f)=c.\]
Clearly, $l_R(T)=c$ iff $f\in J_f$. On the other hand, it is proved in \cite[Theorem 4]{1} that {\it $l_R(T)=c$ iff $R$ is quasi-homogeneous}. Therefore, we obtain the following result.

\begin{theorem}
    \label{risler}
    Let $k$ be an algebraically closed field of characteristic zero and $f\in k^{[[2]]}$ an irreducible power series of order $\ge 2$. Let $R:=\frac{k^{[[2]]}}{(f)}$, and let $\Omega_{R/k}$ be the module of universally finite K\"ahler differentials of $R$ over $k$. Let $T$ be the torsion submodule of $\Omega_{R/k}$. If $J_f$ is the Jacobian ideal of $f$ and $\mathfrak C_R$ is the conductor ideal of $R$, then the following statements are equivalent.
    \begin{enumerate}[(i)]
        \item $f$ is quasi-homogeneous, i.e., there exist a pair of variables $X,Y\in k^{[[2]]}$ and relatively prime positive integers $m,n$ such that $(f)=(Y^n-X^m)$.
        \item There exist a pair of variables $U,V$ and relatively prime positive integers $m,n$ such that $f=V^n-U^m$.
        \item $f$ is contained in the Jacobian ideal $J_f$.
        \item $l_R(T)=c$, where $c:=l_R\big(\frac{\overline{R}}{\mathfrak C_R}\big)$ and $\overline{ R}$ is the integral closure of $R$ in its field of fractions.
    \end{enumerate}
\end{theorem}

\begin{proof}
    We only show that (i) and (ii) are equivalent, as the equivalence of the remaining statements follows from the above discussion.\\
    It is obvious that (ii)$\implies$(i). To prove the converse, let $X,Y$ be a pair of variables and $m,n$ relatively prime positive integers such that $(f)=(Y^n-X^m)$. Then there exists a unit $u\in \big(k^{[[2]]}\big)^*$ such that $f=u(Y^n-X^m)$. Since $k$ is algebraically closed, $k^{[[2]]}$ contains an $n$-th root and an $m$-th root of $u$, say $\sqrt[n]{u}$ and $\sqrt[m]{u}$. If we set $U:=\sqrt[m]{u}X$ and $V:=\sqrt[n]{u}Y$, then it is easy to see that $U,V$ is a pair of variables such that $f=V^n-U^m$.
\end{proof}

\begin{remark}
K. Saito proved in \cite[Theorem 4.1]{2} that {\it if $f$ is a convergent power series ring in $n$ variables over $\mathbb C$ such that $\{f=0\}$ defines an isolated singularity at the origin, then $f\in J_f$ iff $f$ becomes a quasi-homogeneous polynomial after an analytic change of co-ordinates}. The arguments given in \cite{2} also work for formal power series rings in $n$ variables over an algebraically closed field $k$ of characteristic zero, except \cite[Lemma 4.2]{2} where analytic methods are used. The above theorem may be treated as an analogue of Saito's result for a formal power series ring in two variables.\\ Here we should mention that the definition of a quasi-homogeneous polynomial, as given in \cite{2}, is different from our definition. However, they are equivalent for irreducible polynomials in two variables. The verification is left as an exercise for the interested reader.
    
\end{remark} 

Now we discuss parametric representations of an irreducible plane algebroid curve. Our treatment closely follows that of \cite{1}.\\
Let $f\in k^{[[2]]}$ be an irreducible power series of order $n>1$. Then, as described in the {\it introduction}, $R:=\frac{k^{[[2]]}}{(f)}$ has a parametric representation, called a {\it Puiseux parametric representation}, of the form
\[x=t^n,\ y=t^m+\sum_{i>m}c_it^i,\]
where $m$ is not divisible by $n$ and $c_i\in k$ for all $i$. 
%In other words, $R=k[[x,y]]\cong k[[t^n,t^m+\sum_{i>m}c_it^i]]\subseteq k[[t]]\cong \overline R$. 
Since $t$ is contained in the field of fractions of $R$, if $J:=\{m\}\cup \{i>m\ |\ c_i\neq 0\}$ then $\text{g.c.d.}(\{n\}\cup J)=1$.\\
%Recall that $v(R)$, value semi-group of $R$, is defined as \[v(R):=\{v(r)\ |\ r\in R\setminus\{0\}\}=\{\text{ord}_t(r)\ |\ r\in r\setminus \{0\}\},\] where $v$ is the valuation of $\overline{R}$. Clearly, $v(R)$, which is a submonoid of $(\mathbb N,+)$, contains the submonoid generated by $m$ and $n$, which we denote by $\mathbb N_{m,n}$. 
If $\mathfrak C_R$ is the conductor ideal of $R$ and $c:=l_R\big(\frac{ \overline{R}}{\mathfrak C_R}\big)$, then $c$ is the smallest non-negative integer such that $c+\mathbb N:=\{c+i\ |\ i\in \mathbb N\}\subseteq v(R)$, the value semi-groups of $R$. With this set-up, the following results are proved by O. Zariski in \cite[Theorem 4 and the subsequent discussion]{1}.

\begin{theorem}
    Let $k$ be an algebraically closed field of characteristic zero and $f\in k^{[[2]]}$ an irreducible power series of order $n>1$. Then $R:=\frac{k^{[[2]]}}{(f)}$ has a parametric representation of the form 
    \[x=t^n,\ y=t^m+\sum_{i=1}^{c/2}b_it^{e_i}\ ,\]
    called a {\it short (parametric) representation}, where $m>n$ is a positive integer, which is not a multiple of $n$, $b_i\in k$ and $m<e_1<e_2<\ldots <e_{c/2}$ are positive integers such that $e_i \notin v(R)$, the value-semigroup of $R$, for all $i=1,\ldots,c/2$. In particular, $e_1.\dots,e_{\frac{c}{2}}\in \{m+1,\dots,c-1\}$. %Given such a short representation of $f$, 
    Also, the following assertions hold.
    \begin{enumerate}[(i)]
    \label{Z0}
        \item The positive integers $n,m$ are invariants of $R$.
        \item $f$ is quasi-homogeneous iff $y=t^m$, i.e., iff $b_i=0$ for all $i=1,\dots,\frac{c}{2}$, for some short representation of $R$.%The power series $f$ is quasi-homogeneous iff $b_i=0$ for all $i$.
        \item $f$ is not quasi-homogeneous iff $R$ has a short representation of the form 
        \[x=t^n,\ y=t^m+bt^\lambda+\sum_{i=2}^{\frac{c}{2}}b_it^{e_i},\]
        where $b:=b_1$ is nonzero and $\lambda:=e_1$ is such that $\lambda+n$ is not contained in $v(R)$. In this case, $\lambda$ is an analytic invariant of $R$.
        \item If $d\in v(R)\setminus \mathbb N_{m,n}$, where $\mathbb N_{m,n}:=\{am+bn\ |\ a,b\in \mathbb N\}$, then% is the set of all non-negative linear combinations of $m,n$, then
        \[d\ge \frac{mn}{\text{g.c.d.}(m,n)}+\lambda-m.\]
        In particular, since $n$ does not divide $m$, we have $d>2m$.
    \end{enumerate}
\end{theorem}

\begin{remark}
    Here we should mention that there is a typo in the proof of Theorem 4 of \cite{1}. On page 784, it should be `$t=\tau-\frac{a}{n}t^{jm-n+1}+\dots$', and not `$t=\tau-\frac{a}{n}t^{jm+n+1}+\dots\ $'.
\end{remark}

Next we state some results for the value semi-group  of $R:=\frac{k^{[[2]]}}{(f)}$, where $f\in k^{[[2]]}$ is an irreducible power series of order $n\ge 2$. Let $X,Y$ be a pair of variables of $k^{[[2]]}$. Then $f$ can be written as an infinite formal sum $f=\sum_{i=n}^\infty f_i$, where each $f_i$ is zero or a homogeneous polynomial of degree $i$ in $X,Y$, and $f_n\neq 0$. With $k$ being algebraically closed, we can write $f_n$ as a product of $n$ linear homogeneous polynomials. Since $f$ is irreducible, it follows from \cite[exercise 5.14(b) of Chapter 1]{Hart} that $f_n$ cannot have two or more distinct prime factors. Therefore, after a suitable change of variables, we may assume that $f_n=Y^n$. Then by the {\it Weierstrass preparation theorem}, there exist a polynomial $g\in (k[[X]])[Y]$ of the form $g=Y^n+a_1Y^{n-1}+\dots+a_{n-1}Y+a_n$, where $a_i\in Xk[[X]]$ for all $i$, and a unit $u\in (k[[X,Y]])^*$ such that $f=ug$. Since $f=ug$, both $f$ and $g$ have the same order. Therefore, the order of $g$ is also $n$ and $f_n=u(0,0)\cdot g_n$, which implies that $g_n=Y^n$ and $u(0,0)=1$. In addition, it is clear that $a_i\in X^{i+1}k[[X]]$ for all $i=1,\dots,n$. Since $(f)=(g)$, $R=\frac{k[[X,Y]]}{(g)}$; therefore, replacing $f$ with $g$, we can assume at the outset that $f\in k[[X]][Y]$ is a monic polynomial in $Y$ of degree $n$, and the coefficient of $Y^i$ has an $X$-order $\ge n-i+1$ for all $0\le  i<n$. With this set-up, we state the following results, due to S. S. Abhyankar and T. T. Moh, which can be found in \cite[Chapters 1--3]{A} and \cite[Corollary 7.2]{A-M}.

\begin{theorem}
    \label{AM}
    Let $k$ be an algebraically closed field of characteristic zero, and let $f\in k[[X]][Y]$ be a monic polynomial of degree $n\ge 2$ in $Y$, which is irreducible as an element of $k[[X,Y]]$. Let $R:=\frac{k[[X,Y]]}{(f)}$, and let $\overline R\cong k[[t]]$ be the integral closure of $R$. Suppose $R$ has a parametric representation $x(t):=t^n$ and $y(t):=t^m+\textit{higher-degree terms}$, i.e., $f\big(t^n,y(t)\big)=0$, where $m>n$ is not a multiple of $n$. Let \[J:=\{ i \in \mathbb{N} \ |\ \text{the coefficient of $t^i$ in $y(t)$ is nonzero}\}.\] Since $k[[t]]$ is the integral closure of $k[[t^n,y(t)]]$ in its field of fractions, it is clear that $\text{g.c.d.}(\{n\}\cup J)=1$. We define two sequences of positive integers $(m_i)$ and $(d_i)$ as follows :\\
    Let $m_0:=n, m_1:= \text{min }J=m, d_1=n$; and $d_i:= \text{g.c.d.}(d_{i-1}, m_{i-1}), m_i:=\text{min }\{ j \in J \ |\ d_i \text{ does not divide }j\}$, for all $i\ge 2$. Then there exists a unique non-negative integer $h$ such that $d_{h+1}=1$, and we define a finite sequence of positive integers $r_0, \dots , r_h$, as $r_0:= m_0 =n$, and \[r_i = \frac{m_1d_1+\sum_{l=2}^i (m_l -m_{l-1})d_l}{d_i}\ ,\hspace{5mm}  \text{for $1 \leq i \leq h$.}\] 
    Then $r_0 < r_1 < \dots < r_h$, and $r_0, \dots , r_h$ generate $v(R)$, the value semi-group of $R$.
\end{theorem}

\begin{remark}
\label{rem}
    Let $f\in k[[X,Y]]$ be an irreducible power series of order $n\ge 2$ and $R=k[[x,y]]:=\frac{k[[X,Y]]}{(f)}$ be the plane algebroid curve defined by $f$. Let $x=t^n,y=t^m+\textit{higher-degree terms}$ be a parametric representation of $R$. Let $v(R)$ be the value semi-group of $R$, and let $c$ be the length of the conductor ideal of $R$. With this set-up, we make a few elementary observations.
    \begin{enumerate}[(a)]
        \item Since $v(R)$ contains all positive integers starting from $c$, $\text{g.c.d.}(v(R)\setminus\{0\})=1$. Therefore, if $v(R)=\mathbb N_{m,n}$, then $m$ and $n$ must be relatively prime. Conversely, if $\text{g.c.d.}(m,n)=1$, then, with the notation of {\it theorem \ref{AM}}, we have $d_2=1$, which implies that $v(R)$ is generated by $r_0:=n$ and $r_1:=m$. Therefore, it follows from {\it theorem \ref{AM}} that $v(R)$ is generated by $m$ and $n$ iff they are relatively prime.
        \item Suppose $m,n$ are relatively prime. Then the largest positive integer that cannot be written as a non-negative linear combination of $m$ and $n$ is $mn-m-n$. Since $v(R)=\mathbb N_{m,n}$ (by (a)) and $c$ is the smallest positive integer such that $c+\mathbb N\subseteq v(R)$, it follows that if $m$ and $n$ are relatively prime, then $c=mn-m-n+1$.
        \item Since $R$ does not contain any element of $t$-order $c-1$, if $a(t)\in \overline{R}=k[[t]]$ is an element of order $c-n$, then $a(t)\frac{d}{dt}(t^n)\notin R$. Therefore, if the $t$-order of $a(t)$ is $c-n$, then $a(t)\frac{d}{dt}$ does not induce a derivation of $R$.
        \item If $n=2$ and $m\ge 3$ is an odd positive integer, then $c=2m-m-2+1=m-1$. Consequently, it follows from {\it theorem \ref{Z0}\textcolor{blue}{(iii)}} that $f$ is quasi-homogeneous; in fact, there exists a pair of variables $U,V\in k[[X,Y]]$ such that $f=V^2-U^m$.
        \item If $f$ is quasi-homogeneous, then $f$ has a parametric representation of the form $x=t^n$ and $y=t^m$. Since $f$ is irreducible, $\text{g.c.d.}(m,n)=1$; and therefore, it follows from (a) that $v(R)=\mathbb N_{m,n}$. However, the converse is false, i.e., $f$ may not be quasi-homogeneous even if $m,n$ are relatively prime. We can use {\it theorem \ref{Z0}\textcolor{blue}{(iii)}} to construct a plethora of such examples. For example, if $n<m$ are relatively prime positive integers, then for every positive integer $\lambda>m$ that satisfies $\lambda+n\notin v(R)$, a plane algebroid curve with a parametric representation $x=t^n,y=t^m+bt^\lambda$ is not quasi-homogeneous for all nonzero $b\in k$. In particular, since $n,m$ are invariants of $R$, \[\{x=t^3,y=t^p+t^{2p-6}\}_{p\ge 7\text{ is not divisible by }3}\] gives an infinite family of pair-wise non-isomorphic plane algebroid curves which are not quasi-homogeneous.
    \end{enumerate}
\end{remark}

\section{Main result}
\label{sec:3}
	
\iffalse	

\fi

Now we prove our main result. 

\begin{theorem}
\label{main}
Let $f \in k[[X, Y]]$ be an irreducible power series of order $n\ge 2$, which is not quasi-homogeneous. Let $R=k[[x,y]]:=\frac{k[[X,Y]]}{(f)}$, and let $D$ be a nonzero $k$-derivation of $R$. Let $ k[[t]] $ be an integral closure of $k[[x, y]]$.
%which has a short presentation $k[[t^n,t^m+c_1t^{m_1}+\dots+c_st^{m_s}]]\subseteq k[[t]]$ as in Lemma \ref{l}. 
Suppose $D$ is induced by a $k$-derivation $a(t)\frac{d}{dt}:k[[t]]\to k[[t]]$ for some $a(t)\in k[[t]]$. If $r:=\text{ord}_t(a)$, then $r\ge n+1$.
%, then the $t$-order of $a(t)$ is at least $n+1$.%, then $r\ge n$.
    
\end{theorem}

\begin{proof}

First we show that the $t$-order of $a(t)$ cannot be smaller than $n$.\\ 
If possible, let $r<n$. We consider two cases, where $r=1$ and where $r>1$.\\
%If $m > 2n$, then $r =1$ is the only possibility, as the order of $a(t)\frac{d}{dt}(t^n)$ is $n+r-1 <2n <m$, which implies that $n$ is the only possible value of $n+r-1$. Note that $ r=1$ is possible even if $n < m < 2n$. But in that case, it is not the only possibility. To see this, suppose $r>1$. 
%; and, in particular, $n<m<2m$.\\
{\bf Case 1 : $r =1$}\\
By slight abuse of notation, we also use $D$ to denote the $k$-derivation $ a(t)\frac{d}{dt}$ of $k[[t]]$. With the order of $f$ being $n$, it follows from the discussion before {\it theorem \ref{AM}} that after a suitable change of variables, we may assume that $f$ is of the form $Y^n + f_1(X)Y^{n-1} + \dots + f_n(X)\in k[[X]][Y]$, where the $X$-order of each $ f_i$ is at least $i+1$. Since $D $ induces a derivation of $R$, $f_xD(x) + f_yD(y)=0$, where $f_x,f_y$ are images of the usual partial derivatives $f_X,f_Y$ under the natural projection $k[[X,Y]]\to k[[x,y]]$.  As $\overline R\cong k[[t]]$, we identify $R$ with a subring of $k[[t]]$ using a parametric representation $x=t^n$ and $y=t^m + \textit{higher-degree terms}$, where $m>n$ is a positive integer, which is not a multiple of $n$.\\ 
%A derivation $D$ of $ R $ can be written as $a(t)\frac{d}{dt}$ when considered as a derivation of $k[[t^n, t^m + \dots]]$. 
If the order of $a(t)$ is $1$, then $a(t)\frac{d}{dt}(x)$ has order $n$ and $a(t)\frac{d}{dt}(y)$ has order $ m $. Therefore, if we write $D(x)$ and $D(y)$ as power series in $x,y$, then $D(x) = \lambda_1x + \varphi$  and $D(y) = \lambda_2y + \psi$, where $\lambda_1,\lambda_2\in k$ and nonzero constants and $\varphi,\psi$ are power series in $x,y$ of order at least $2$.
%(Note that according to our assumption $n$ does not divide $m$.). 
As $\big(\lambda_1x + \varphi(x,y)\big)f_x + \big(\lambda_2y + \psi(x,y)\big)f_y =0$, the power series $\big(\lambda_1X + \varphi(X,Y)\big)f_X + \big(\lambda_2Y + \psi(X,Y)\big)f_Y\in k[[X,Y]]$ is a multiple of $f$. From the description of $f$, it is clear that the order of $f_X$ is at least $n$ and the order of $f_Y \text{ is } n-1$. Therefore, the order of $\big(\lambda_1X + \varphi(X,Y)\big)f_X + \big(\lambda_2Y + \psi(X,Y)\big)f_Y$ is $n$, which implies that $\big(\lambda_1X+\varphi(X,Y)\big)f_X + \big(\lambda_2Y + \psi(X,Y)\big)f_Y=uf$ for some unit $u\in (k[[X, Y]])^*$. Therefore, $f$ is contained in $J_f$, the Jacobian ideal of $f$; consequently, it follows from {\it theorem \ref{risler}}, that $f$ is quasi-homogeneous, which is a contradiction. Therefore, we conclude that $r>1$.\\
{\bf Case 2 : $1<r< n$}\\ %, $n=2(r-1)$ and $m=3(r-1)$}\\
Since $f$ is not quasi-homogeneous, let $x=t^n$, $y=t^m+bt^{\lambda}+\textit{higher-degree terms}$ be a short presentation of $R$, as in {\it theorem \ref{Z0}\textcolor{blue}{(iii)}}. It follows from {\it theorem \ref{Z0}\textcolor{blue}{(iv)}} %{\it Lemma 5} of Zariski's paper 
that if $r\in R$ is a nonzero element whose order cannot be written as a non-negative linear combination of $m$ and $n$, then $\text{ord}_t(r)>2m$. Now, the orders of $a(t)\frac{d}{dt}(x)$ and $a(t)\frac{d}{dt}(y)$ are $n+r-1$ and $m+r-1$, respectively. Since $r>1$ and $n+r-1,m+r-1$ are smaller than $2m$, it is easy to see that $n+r-1=m$ and $m+r-1=2n$. Solving the equations, we get $n=2(r-1)$ and $m=3(r-1)$. Moreover, we can assume that $r\ge 3$; for if $r=2$ then $n=2$ and $m=3$, which contradicts our assumption that $r<n$.\\ 
Note that $\mathbb N_{m,n}=\{s(r-1)\ |\ s\in \mathbb N\setminus\{1\}\}$. 
%We know that, in this case, $n=2(r-1)$ and $m=3(r-1)$. 
%Let $q$ be the smallest positive integer of $v(R)$ that cannot be written as a non-negative linear combination of $m$ and $n$.
%$\mathcal S$ be the set of nonzero elements of $v(R)$, and let $q$ be the smallest element of $\mathcal S$ that cannot be written as a nonnegative linear combination of $m$ and $n$. 
The $t$-order of $y^2-x^3$ is $\lambda+m$, which implies that $\lambda+m\in v(R)$. If $\lambda+m\in \mathbb N_{m,n}$, then $\lambda +m=r'(r-1)$ for some $r'\ge 7$. But then $\lambda=(r'-3)(r-1)$ can also be written as a non-negative linear combination of $m$ and $n$, which contradicts the fact that $\lambda\notin v(R)$. Therefore, $\lambda+m$ cannot be written as a non-negative linear combination of $m$ and $n$; in fact, since $n=2(r-1)$ and $m=3(r-1)$, it follows from {\it theorem \ref{Z0}\textcolor{blue}{(iv)}} that $\lambda+m$ is the smallest element of $v(R)\setminus \mathbb N_{m,n}$. Now, the $t$-order of $a(t)\frac{d}{dt}(y^2-x^3)$ is $\lambda+m+r-1=\lambda+2n$. As $\lambda+2n>7(r-1)$, if $\lambda+2n\in \mathbb N_{m,n}$ then $\lambda\in \mathbb N_{m,n}$, which is a contradiction. Therefore, if $\lambda+2n$ cannot be written as a non-negative linear combination of $m$ and $n$. %, which is not a non-negative linear combination of $m, n$. 
If we can show that $\lambda+2n\notin v(R)$, then we will get a contradiction.\\
%We want to show that $m_1+2n \notin G$, thereby getting a contradiction.
%If $q< q^{\prime} \in G$ is such that $q^{\prime}$ is not a  combination of $m, n$, we show that $q^{\prime} - q \geq n$.\\ 
A positive integer $n_0 \ge 2(r-1)$ is divisible by $r-1$ if and only if it is a non-negative linear combination of $m$ and $n$.
% 
% 
%Write $x = t^n, y = t^m + c_1t^{m_1} + \text{higher degree terms}, m_1 > m$. 
If $x^iy^j\in R$ is a non-constant monomial, then the $t$-order of $x^iy^j$ is divisible by $r-1$. Therefore, an element of $R$ whose $t$-order is not divisible by $r-1$ must contain at least two monomials.\\%be of the form $\mu_1x^{i_1}y^{j_1} + \mu_2x^{i_2}y^{j_2} +\dots $, i.e., it cannot be a monomial.%If t-order of such an element is not of the form \alpha m + \beta n then at least two different monomial must have same t-order. \noindent
If $h \in R\setminus k$, then we can write $h$ as \[h(x ,y) = h_1(x ,y) + h_2(x, y) + h_3(x,y)+\dots\ ,\] where, for each $i$, $h_i$ is the sum of all monomials of $h$ that have some particular $ t $-order, say $s_i$, such that $s_1 < s_2 < s_3<\dots$ (Note that every $h_i$ is a polynomial in $x,y$, and $i<j$ does not imply $\text{ord}_t(h_i)< \text{ord}_t(h_j)$ in general.). If the $ t $-order of $h$ is not a non-negative linear combination of $m,n$, then the same is true for the $ t $-order of $h_1$; and this is possible only if $h_1$ contains at least two monomials or, equivalently, only if $s_1$ can be written as a non-negative linear combination of $m$ and $n$ in more than one way. The image of a typical monomial like $x^iy^j$ is \[\big(t^{2(r-1)}\big)^i \big(t^{3(r-1)}\big)^j\big(1 + c_1t^{\lambda -m} + \textit{higher-degree terms}\big)^j\]
\[= t^{(2i +3j)(r-1)}(1 + jc_1t^{\lambda -m} +\textit{higher-degree terms}) \ \ \ldots\ldots \ (*).\] 
%N
We know that $\lambda+m$ is the smallest positive integer of $v(R)$ that is not a multiple of $r-1$. Since $\lambda+2n$ is not a multiple of $r-1$, to prove that $\lambda+2n\notin v(R)$, it suffices to show that if $q$ is the smallest positive integer of $v(R)\setminus\{\lambda+m\}$ which is not a multiple of $r-1$, then $q>\lambda+2n$.\\ To prove it, let $\psi\in R$ be an element whose $t$-order is $q$. As above, we can write $\psi$ as
\[\psi(x ,y) = \psi_1(x ,y) + \psi_2(x, y) + \psi_3(x,y)+\dots\ ,\] where, for each $i$, $\psi_i$ is the sum of all monomials of $\psi$ that have the same $ t $-order, say $d_i$, such that $d_1 < d_2 < d_3<\dots$ . It follows from $(*)$ that, for each $i$, if the $t$-order of $\psi_i$ is not equal to $d_i$, then it is at least $d_i+\lambda-m$. As the $t$-order of $\psi$, namely $q$, is not a multiple of $r-1$, the same is true for the $t$-order of $\psi_1$, which implies that $\psi_1$ is a sum of two or more monomials. Therefore, $d_1$ can be written as non-negative linear combinations of $m,n$ in more than one way. The smallest positive integer for which this is possible is $6(r-1) = 3n = 2m$.\\ 
If $d_1=6(r-1)$, then $\psi_1(x,y)=aX^3+bY^2$ for some nonzero $a,b\in k$. In this case, the $t$-order of $\psi_1$ is $\lambda+m$, since it is not a multiple of $r-1$. For all $i\ge 2$, if a monomial in $t$ whose degree is not a multiple of $r-1$ appears in the image of $\psi_i$, then it follows from $(*)$ that the $t$-order of that monomial is at least $d_i+\lambda-m\ge 7(r-1)+\lambda-m=\lambda+4(r-1)>\lambda+m$. As a result, the $t$-order of $\psi$ would be $\lambda+m<q$, which is a contradiction. Therefore, $d_1>6(r-1)$ is a positive integer that can be written as non-negative linear combinations of $m,n$ in more than one way. Since $7$ can be uniquely written as a non-negative linear combination of $2$ and $3$, it follows that $d_1\ge 8(r-1)$. Therefore, by $(*)$, $q\ge 8(r-1)+\lambda-m=\lambda+5(r-1)>\lambda+2n$, which shows that $\lambda+2n\notin v(R)$. Hence, we conclude that $r$ cannot be smaller than $n$.\\
Finally, we show that the $t$-order of $a(t)$ cannot be equal to $n$.\\
If possible, let $a(t)\in k[[t]]$ be a nonzero power series of order $n$ such that $D:=a(t)\frac{d}{dt}$ induces a derivation of $R$. Since $x=t^n$ and $y=t^m+\textit{higher-degree terms}$, the $t$-orders of $D(x)$ and $D(y)$ are $2n-1$ and $m+n-1$, respectively. It follows from {\it theorem \ref{Z0}\textcolor{blue}{(iv)}} that $2n-1$ is a non-negative linear combination of $m$ and $n$, which implies that $m=2n-1$. Since $m+n-1=3n-2<2m$, we can again apply {\it theorem \ref{Z0}\textcolor{blue}{(iv)}} to conclude that $3n-2$ is a non-negative linear combination of $m$ and $n$. Therefore, $3n-2=2n$, which implies that $n=2$ and $m=3$. Consequently, by {\it remark \ref{rem}\textcolor{blue}{(d)}}, $f$ is quasi-homogeneous; in fact, $f=V^2-U^3$ for a suitable pair of variables $U,V\in k[[X,Y]]$. This is a contradiction. Therefore, we conclude that $r\ge n+1$, and this completes the proof.
%a$3$, $q'>7(r-1)$. Therefore, $q'+m_1-m>m_1+4(r-1)=m_1+2n$, which completes the proof.
\end{proof}

\begin{remark}
\label{rem1}
The lower bound given on the $t$-order of $a(t)$ in {\it theorem \ref{main}} is sharp. To show this, for each $n\ge 3$, we construct a family of non-quasi-homogeneous plane algebroid curves such that every member of the family has multiplicity $n$ and admits a derivation of the form $a(t)\frac{d}{dt}$, where $a(t)$ has a $t$-order $n+1$. We consider two cases, where $n=3$ and where $n\ge 4$. We follow the set-up and notation as in {\it theorem \ref{main}}.\\
{\bf $n=3$ :} Let $R$ be the plane algebroid curve with parametric representation $x=t^3, y=t^m+bt^{2m-6}$, where $m\ge 7$ is not divisible by $3$ and $b\in k$ is a nonzero constant. Let $D:k[[t]]\to k[[t]]$ be the $k$-derivation defined as $D:=(t^4+\alpha t^{m-2}+b\alpha t^{2m-8})\frac{d}{dt}$, where $\alpha:=\frac{b(6-m)}{m}$. In this example, $\lambda=2m-6$ and $c=2m-2$. As $\lambda+3=2m-3\notin v(R)$, $R$ is not quasi-homogeneous by {\it theorem \ref{Z0}\textcolor{blue}{(iii)}}. It is easy to check that $D(x)=3x^2+3\alpha y$, and $D(y)-mxy$ has a $t$-order $\ge 2m-2$. Since $R$ contains $t^i$ for all $i\ge 2m-2$, it follows that $D(x),D(y)-mxy\in R$. Therefore, $D$ induces a derivation of $R$.\\
{\bf $n\ge 4$ :} Let $R$ be the plane algebroid curve with parametric representation $x=t^n, y=t^m+bt^{mn-m-2n}$, where $m>n$ is relatively prime to $n$ and $b\in k$ is a nonzero constant. Let $D:k[[t]]\to k[[t]]$ be the derivation given by $\big(t^{n+1}+\frac{b}{m}\big(2(m+n)-mn\big)t^{mn-2m-n+1}\big)\frac{d}{dt}$. In this example, $\lambda=mn-m-2n$ and $c=mn-m-n+1$. As $\lambda+n=mn-m-n\notin v(R)$, $R$ is not quasi-homogeneous by {\it theorem \ref{Z0}\textcolor{blue}{(iii)}}. It is easy to check that $D(x)-nx^2-\frac{nb(2m+2n-mn)}{m}y^{n-2}$ and $D(y)-mxy$ have $t$-orders $\ge c$. Since $R$ contains $t^i$ for all $i\ge c$, it follows that $D(x)-nx^2-\frac{nb(2m+2n-mn)}{m}y^{n-2},D(y)-mxy\in R$. Consequently, $D$ induces a derivation of $R$.
%with the set-up and notation as in {\it theorem \ref{main}}, {\it let $R$ be the plane algebroid curve with parametric representation $x=t^4$, $y=t^7+bt^{13}$, where $b\in k$ is a nonzero element; then the derivation $D:k[[t]]\to k[[t]]$, defined as $D:=(t^5-\frac{6b}{7}t^{11})\frac{d}{dt}$, induces a derivation of $R$}.\\In this example, we have $n=4,m=7,\lambda=13$ and $c=18$. Therefore, $t^i\in R$ for all $i\ge 18$. Since $\lambda+n=17\notin v(R)$, it follows from {\it theorem \ref{Z0}\textcolor{blue}{(iii)}} that $R$ is not quasi-homogeneous. Now, an easy computation shows that $D(x)-4x^2+\frac{24b}{7}y^2$ and $D(y)-7xy$ have $t$-orders $\ge 20$, implying that $D(x),D(y)\in R$. Here, the $t$-order of $a(t)$ is $5$, which is much smaller than $c-n+1=15$.   
\end{remark}

\section{A result for the value semi-group $v(R)$}
\label{sec:4}

In this section, we prove a result for the value semi-group of a plane algebroid curve, which is inspired by case 2 in the proof of {\it theorem \ref{main}}. Our result, proved using the generators of $v(R)$ as described in \cite{A-M}, is more general and of independent interest.

As usual, let $f\in k[[X,Y]]$ be an irreducible power series of order $n\ge 2$ and let $R=k[[x,y]]:=\frac{k[[X,Y]]}{(f)}$ be the plane algebroid curve defined by $f$. In view of the discussion before {\it theorem \ref{AM}}, we may assume that $f$ is of the form $f=Y^n+a_iY^{n-1}+\dots+a_{n-1}Y+a_n\in k[[X]][Y]$, where $a_i\in X^{i+1}k[[X]]$ for all $i$. Suppose $f$ is not quasi-homogeneous. Then, by {\it theorem \ref{Z0}\textcolor{blue}{(iii)}}, there exists a short parametric representation of $R$ as 
\[x=t^n,\ y=t^m+bt^\lambda+\sum_{i=2}^{\frac{c}{2}}b_it^{e_i},\]
where $b\in k$ is nonzero and $\lambda+n\notin v(R)$. However, $\lambda+2n$ can be contained in $v(R)$, as can be seen from the example of the plane algebroid curve with short representation $x=t^4,y=t^7+t^{13}$.

We will prove that if, moreover, it is assumed that $\text{g.c.d.}(m,n)$ does not divide $\lambda$, then $\lambda+2n\notin v(R)$ (our assumption implies that $m$ and $n$ are not relatively prime). First, we prove a more general result for value semi-groups.

\begin{proposition}
\label{4prop}
Let $f\in k^{[[2]]}$ be an irreducible power series of order $n\ge 2$ and let $R:=\frac{k^{[[2]]}}{(f)}$ be the plane algebroid curve defined by $f$. Without loss of generality, we assume that $f=Y^n+a_1Y^{n-1}+\dots+a_{n-1}Y+a_n\in k[[X]][Y]$ for a pair of variables $X,Y\in k^{[[n]]}$ such that $a_i\in X^{i+1}k[[X]]$ for all $i=1,\dots,n$. We assume that $f$ is not quasi-homogeneous, so $R$ has a parametric representation of the form $x=t^n,y=t^m+bt^{\lambda}+\textit{higher-degree terms}$, where $b$ is nonzero and $\lambda+n\notin v(R)$. Suppose $v(R)$ is not generated by $m,n$ or, equivalently, $\text{g.c.d.}(m,n)\neq 1$. If $(m_i)$ and $(d_i)$ are the sequences of positive integers as defined in {\it theorem \ref{AM}}, then $m_2+2n\notin v(R)$.  
    
\end{proposition}

\begin{proof}
Keeping the same set-up and notation as in {\it theorem \ref{AM}}, we explicitly describe the first few terms of the two sequences $(m_i)$ and $(d_i)$, and $r_0,\dots,r_h$, the generating set of $v(R)$, which will be used in the proof. Recall that 
\[
m_0:=d_1:=n, m_1:=m,d_2:=\text{g.c.d.}(n,m_1), d_3:=\text{g.c.d.}(n,m_1,m_2),\]
\[r_0:=n,r_1:=m_1,r_2:=\frac{nm_1}{d_2}+m_2-m_1 \text{,  }\ \text{and } r_3:=\frac{nm_1+(m_2-m_1)d_2+(m_3-m_2)d_3}{d_3}\ .\]
According to our assumption, $d_2:=\text{g.c.d.}(n,m_1)>1$. We consider two cases, where $d_3=1$ and where $d_3>1$.\\
{\bf Case 1 : $d_3=1$}\\
Since $d_2>d_3=1$, it follows from {\it theorem \ref{AM}} that $h=2$ and $v(R)$ is generated by $r_0=n,r_1=m_1$ and $r_2=\frac{nm_1}{d_2}+m_2-m_1$. Since $n$ does not divide $m_1$, $n\ge 2d_2$.\\ 
If $n\ge 3d_2$, then $r_2\ge 3m_1+m_2-m_1=2m_1+m_2>m_2+2n$. Therefore, in this case, if $m_2+2n\in v(R)$, then $m_2+2n$ can be written as a non-negative linear combination of $n$ and $m_1$. But then $m_2$ is divisible by $d_2$, which contradicts the choice of $m_2$.\\ 
Therefore, we can assume that $n=2d_2$. Then $r_2=m_2+m_1$. If $m_2+2n\in v(R)$, then there exist non-negative integers $a_1,a_2,a_2$ such that $m_2+2n=a_1n+a_2m_1+a_3(m_2+m_1)$. Since $m_2$ is not divisible by $d_2$, $a_3$ must be nonzero. If $a_3>1$, then the right hand side is bigger than $m_2+2n$, which implies $a_3=1$. Therefore, $2n=a_1n+(a_2+1)m_1$. Since $m_1>n$, $a_2$ must be $0$, which implies that $2n=a_1n+m_1$. So, $a_1=0$ or $1$. Since $n$ does not divide $m_1$, $a_1$ cannot be $0$; and if $a_1=1$ then $m_1+n>2n$. In either case, we get a contradiction. Therefore, if $d_3=1$ then $m_2+2n\notin v(R)$.\\
{\bf Case 2 : $d_3>1$}\\
Since $d_1:=n>d_2>d_3>1$ and each $d_i$ is a multiple of $d_{i+1}$, $d_2\ge 2d_3$ and $n\ge 4d_3$. Therefore, \[r_3:=\frac{m_1n+(m_2-m_1)d_2}{d_3}+m_3-m_2\ge 4m_1+2(m_2-m_1)+m_3-m_2=2m_1+m_2+m_3>m_2+2n.\]
Since $r_0<r_1<r_2<\dots<r_h$ generate $v(R)$ and $r_3>m_2+2n$, it follows that if $m_2+2n\in v(R)$ then it must be a non-negative linear combination of $r_0=n,r_1=m_1$ and $r_2$. But that is not possible, as already proved in {\it case 1}. \\
Combining the two cases, we conclude that $m_2+2n\notin v(R)$.
\end{proof}

\begin{corollary}
\label{4cor}
Let $f\in k[[X,Y]]$ be an irreducible power series of order $n\ge 2$. Suppose $f$ is not quasi-homogeneous and $R=k[[x,y]]:=\frac{k[[X,Y]]}{(f)}$ has a parametric representation $x=t^n,y=t^m+bt^\lambda+\textit{higher-degree terms}$ as in {\it theorem \ref{Z0}\textcolor{blue}{(iii)}}. If $\text{g.c.d.}(m,n)$ does not divide $\lambda$, then $\lambda+2n\notin v(R)$.    
\end{corollary}

\begin{proof}
Since $\text{g.c.d.}(m,n)$ does not divide $\lambda$,we have $m_2=\lambda$. Therefore, the assertion immediately follows from {\it proposition \ref{4prop}}. 
\end{proof}

\begin{remark}
\label{4rem}
\begin{enumerate}
\item With the set-up and notation as in case 2 of the proof of {\it theorem \ref{main}}, it was shown that $\text{g.c.d}(m,n)=r-1$ does not divide $\lambda$. Therefore, it immediately follows from {\it corollary \ref{4cor}} that $\lambda+2n\notin v(R)$. This gives a more conceptual proof of case 2 of the proof of {\it theorem \ref{main}}.
\item Keeping the set-up and notation as in {\it corollary \ref{4cor}}, we give an example to show that the condition `$\text{g.c.d.}(m,n)$ does not divide $\lambda$' can be dropped in general.\\
Let $R:=k[[x,y]]$ be the plane algebroid curve with parametric representation $x=t^{4d}$, $y=t^{5d}+t^{7d}+t^{11d-1}$, where $d>1$ is a positive integer. Let $\varphi$ be a nonzero power series in $x,y$. If the order of $\varphi$, as a power series in $x,y$, is $\le 1$, then the $t$-order of $\varphi$ is at most $5d$. On the other hand, if the order of $\varphi$, as a power series in $x,y$, is $\ge 3$, then the $t$-order of $\varphi$ is at least $12d$. Finally, if the order of $\varphi$ is $2$ as a power series in $x,y$, then the $t$-order of $\varphi$ is $8d$ or $9d$ or $10d$. Therefore, $7d,11d-1,11d\notin v(R)$, and the parametric representation given is a short representation of $R$, with $n=4d,m=5d$ and $\lambda=7d$. Note that $d=\text{g.c.d.}(m,n)$ divides $\lambda$ and $m_2=11d-1$. Here $\lambda+n=11d\notin v(R)$, but $\lambda+2n=15d\in v(R)$.
\end{enumerate}

%are relatively prime positive integers $>1$ such that $e>34d$. Here $n=7d$, $m=8d$ and $\lambda=34d$. So, $m,n$ do not generate $v(R)$, but $\text{g.c.d.}(m,n)=d$ divides $\lambda$. It is clear that $\lambda+2n=48d\in v(R)$.Let $R$ be the plane algebroid curve with parametric representation $x=t^{7d}$, $y=t^{8d}+t^{34d}+t^e$, where $d,e$ are relatively prime positive integers $>1$ such that $e>34d$. Here $n=7d$, $m=8d$ and $\lambda=34d$. So, $m,n$ do not generate $v(R)$, but $\text{g.c.d.}(m,n)=d$ divides $\lambda$. It is clear that $\lambda+2n=48d\in v(R)$.

\end{remark}

So far, we have seen many examples of plane algebroid curves of multiplicity $n\ (\ge 2)$ that admit derivations of the form $a(t)\frac{d}{dt}$, where the $t$-order of $a(t)$ is $n+1$. However, this is not true in general. Below we construct a class of plane algebroid curves for which this fails. 

\begin{proposition}
    \label{4prop2}
    Let $2\le n<m$ be positive integers such that $m,n$ are not relatively prime and $n$ does not divide $m$. Let $R:=k[[x,y]]:=\frac{k[[X,Y]]}{(f)}$ be an irreducible plane algebroid curve with parametric representation $x=t^n,y=t^m+bt^{m+1}+\textit{higher-degree terms}$, where $b\in k$ is a nonzero constant. If $a(t)\in \overline{R}=k[[t]]$ is a nonzero power series such that $a(t)\frac{d}{dt}$ induces a derivation of $R$, then the $t$-order of $a(t)$ is at least $n+2$. 
    %Then $R$ is not quasi-homogeneous by {\it theorem \ref{Z0}\textcolor{blue}{(iii)}}. 
\end{proposition}

\begin{proof}
    An elementary computation shows that $m+n+1\notin v(R)$, the value semi-group of $R$. Consequently, by {\it theorem \ref{Z0}\textcolor{blue}{(iii)}}, $R$ is not quasi-homogeneous. Suppose $a(t)$ is a nonzero power series such that $D:=a(t)\frac{d}{dt}$ induces a derivation of $R$. By {\it theorem \ref{main}}, the $t$-order of $a(t)$ is at least $n+1$. If $a(t)=t^{n+1}$, then $D(y)-mxy$ has a $t$-order $m+n+1$, which is a contradiction. Therefore, we can write $a(t)$ as $a(t)=t^{n+1}+\alpha t^s+\textit{higher degree terms}$, where $n+1<s$ and $\alpha\in k$ is a nonzero constant. If $s>n+2$, then $D(y)-mxy$ has a $t$-order $m+n+1$, which is a contradiction. Therefore, $s=n+2$. As $D(x)-nx^2\in R$, it follows that $s+n-1=2n+1\in v(R)$. But then $m+1+2n\in v(R)$, which contradicts {\it corollary \ref{4cor}}. Therefore, the $t$-order of $a(t)$ must be at least $n+2$.     
    %Since $s\ge n+2$, it follows that $s+m-1\ge \lambda+n$. If $s+m-1>\lambda+n$
\end{proof}

%For example, if $d>1$ is a positive integer, then one may consider the irreducible plane algebroid curve $R$ with a parametric representation $x=t^{2d},y=t^{3d}+t^{3d+1}$. In this example, $n=2d,m=3d$ and $\lambda =3d+1$. An elementary computation shows that $4d+1,5d+1\notin v(R)$. As $\lambda+n=5d+1\notin v(R)$, by {\it theorem \ref{Z0}\textcolor{blue}{(iii)}}, $R$ is not quasi-homogeneous. If possible, let $a(t):=t^{2d+1}+\alpha t^{2d+2}+\beta t^{2d+3}+\dots\in \overline R=k[[t]]$ be a power series such that $D:=a(t)\frac{d}{dt}$ induces a derivation of $R$. If $a(t)=t^{2d+1}$, then $D(y)-3dxy=3dt^{5d+1}\in R$, which is a contradiction since $5d+1\notin v(R)$. Therefore, $a(t)\neq t^{2d+1}$. If $\alpha\neq 0$, then the $t$-order of $D(x)-2dx^2$ is $4d+1$. Since $4d+1\notin v(R)$, it follows that $\alpha=0$. Now, it is easy to check that the $t$-order of $D(y)-3dxy$ is $5d+1$, which is a contradiction, as $5d+1\notin v(R)$. Therefore, the $t$-order of $a(t)$ cannot be $n+1$. Note that in this example $m$ and $n$ are not relatively prime.\\
The contrasting nature between the examples of {\it remark \ref{rem1}} and the ones given above suggests that it is worth classifying plane algebroid curves with multiplicity $n$ that admit a derivation of the form $a(t)\frac{d}{dt}$, where the $t$-order of $a(t)$ is $n+1$. Since every quasi-homogeneous plane algebroid curve has this property, we can restrict our attention to plane algebroid curves which are not quasi-homogeneous. Also, in the above example $m$ and $n$ are not relatively prime. Therefore, we ask the following question.

{\bf Question.} Let $R=k[[x,y]]:=\frac{k[[X,Y]]}{(f)}$ be an irreducible plane algebroid curve of multiplicity $n\ge 2$. Let $x=t^n,y=t^m+\textit{higher-degree terms}$ be a Puiseux parametric representation of $R$, where $n$ does not divide $m$. Let $\overline{R}\cong k[[t]]$ be the integral closure of $R$ in its field of fractions. If $R$ is not quasi-homogeneous and $m,n$ are relatively prime, then must there exist a nonzero power series $a(t)\in k[[t]]$ of $t$-order $n+1$ such that $a(t)\frac{d}{dt}$ induces a $k$-derivation of $R$?

\begin{Acknowledgement}
The third author would like to acknowledge the financial support provided by the National Board for Higher Mathematics (NBHM), Department of Atomic Energy, Government of India, Ref No: $0203/13(36)/2021-R \& D-II/13163 $.
	 %The third author's research is funded by the National Board for Higher Mathematics (NBHM), Department of Atomic Energy (DAE), Government of India, Ref No: $ 0203/13(36)/2021-R \& D-II/13163 $.
\end{Acknowledgement}

\end{document}